\documentclass{amsart}
\numberwithin{equation}{section}
\newtheorem{theorem}{Theorem}[section]
\newtheorem{corollary}[theorem]{Corollary}

\newtheorem{conjecture}[theorem]{Conjecture}
\newtheorem{lemma}[theorem]{Lemma}
\theoremstyle{definition}

\newtheorem{example}[theorem]{Example}

\theoremstyle{remark}
\newtheorem{remark}[theorem]{\bf\em Remark}

\usepackage{graphicx}
\usepackage{amsfonts}
\usepackage{amssymb}
\usepackage{amsmath}
\usepackage{color}
\usepackage[inner=1in,outer=1in,bottom=1in,top=1in]{geometry}

\begin{document}
\title[Diophantine equations and perturbations of symmetric Boolean functions]{Diophantine equations with binomial coefficients and perturbations of symmetric Boolean functions}

\author{Francis N. Castro}
\address{Department of Mathematics, University of Puerto Rico, San Juan, PR 00931}
\email{franciscastr@gmail.com}

\author{Oscar E. Gonz\'alez}
\address{Department of Mathematics, University of Puerto Rico, San Juan, PR 00931}
\email{oscar.gonzalez3@upr.edu}

\author{Luis A. Medina}
\address{Department of Mathematics, University of Puerto Rico, San Juan, PR 00931}
\email{luis.medina17@upr.edu}

\begin{abstract}
This work presents a study of perturbations of symmetric Boolean functions.  In particular, it establishes a connection between exponential sums of these perturbations and Diophantine equations
of the form
$$ \sum_{l=0}^n \binom{n}{l} x_l=0,$$
where $x_j$ belongs to some fixed bounded subset $\Gamma$ of $\mathbb{Z}$.  The concepts of trivially balanced symmetric Boolean function and sporadic balanced Boolean function are extended to this type of perturbations.
An observation made by Canteaut and Videau \cite{canteaut} for symmetric Boolean functions of fixed degree is extended.  To be specific, it is proved that, excluding the trivial cases,
balanced perturbations of fixed degree do not exist when the number of variables grows.  Some sporadic balanced perturbations are presented.  Finally, a beautiful but unexpected identity between 
perturbations of two very different symmetric Boolean functions is also included in this work.
\end{abstract}

\date{\today}

\maketitle
\section{Introduction}
The theory of Boolean functions is a beautiful area of combinatorics with vast applications to many areas of mathematics as well as outside the discipline.  Examples include electrical engineering, 
the theory of error-correcting codes and cryptography.  In the modern era, efficient implementations of Boolean functions with many variables is a challenging problem due to memory restrictions of current 
technology. Because of this, symmetric Boolean functions are good candidates for efficient implementations.   However, symmetry is a too special property and may imply that these implementations 
are vulnerable to attacks.  
For this reason, we study perturbations of symmetric Boolean functions.  These perturbations, which are the focus of \cite{cm2}, are not longer symmetric.  Nevertheless, the symmetry of the underlying 
function can be exploited in order to make fast calculations, to obtain recurrences, and, as it was done in \cite{cm2}, to obtain information about the asymptotic behavior.

In plenty of applications, especially the ones related to cryptography, it is important for Boolean functions to be balanced. A balanced Boolean function is one for which the number of zeros and the number 
of ones are equal in its truth table.  Balancedness of Boolean functions can be studied from the point of view of exponential sums, as it is done is this article.  This point of view is in fact a very active
area of research.  For some examples, please refer to \cite{sperber, ax, cm1, cm2, cm3, cgm, fspectrum, mm1, mm, mcsk, fdegree}.

The study of balancedness of symmetric Boolean functions is connected to the problem of bisecting binomial coefficients.    A solution $(\delta_0,\delta_1,\cdots, \delta_n)$ to the equation
\begin{equation}
\label{bisec}
 \sum_{l=0}^n x_l \binom{n}{l}=0,\,\,\, x_l \in \{-1,1\},
\end{equation}
is said to give a bisection of the binomial coefficients $\binom{n}{l}$, $0\leq l \leq n.$  The first detailed study of this connection was made by Mitchell \cite{mitchell}.  Other studies include 
Jefferies \cite{jeff} and Sarkar and Maitra \cite{sarkarmaitra}.  In this work, balancedness of the perturbations considered is linked to equation (\ref{bisec}), where the $x_l$'s now 
lie in a bounded subset $\Gamma$ of $\mathbb{Z}$ instead of in $\{-1,1\}$.  The concept of trivially balanced symmetric Boolean function and the concept of sporadic balanced symmetric Boolean function (which
was introduced in \cite{jeff}), are extended to these perturbations. 

A conjecture similar to the one presented in \cite{cls} for elementary symmetric Boolean functions seems to be true for the simplest 
type of the perturbations considered in this study. Also, similar to the case of symmetric Boolean functions, computations suggest that most balanced perturbations are trivially balanced.
This led us to study trivially balanced perturbations in more detail: we showed that once a perturbation of fixed degree is trivially balanced at one point, then it is trivially balanced at infinitely
many points.  In \cite{canteaut}, Canteaut and Videau observed that, excluding the trivial cases, balanced symmetric Boolean 
functions of fixed degree do not exist when the number of variables grows (this was recently proved in \cite{ggz}).  This result is extended to our perturbations, that is, it is proved that, excluding the trivial cases,
balanced perturbations of fixed degree do not exist when the number of variables grows.  Therefore, the search for sporadic balanced perturbations is of interest.  It is in this search 
that an striking identity between the perturbations of two different symmetric Boolean functions is used.  In particular, this identity allow us to obtained two sporadic balanced perturbations for ``the price of one".

This article is divided as follows.  The next section includes some preliminaries that are needed for the work presented in this manuscript.  In section \ref{pertuident}, a 
beautiful but unexpected 
identity between the perturbations of two very different symmetric Boolean functions is proved.  This identity is later used in section \ref{sporadicpert} when the search of sporadic perturbations is considered.  
In section \ref{diophsec}, the study of the link between perturbations and equation (\ref{bisec}) over bounded sets of integers is considered.  Some of the known results about balancedness of 
symmetric Boolean functions and bisections of binomial coefficients are extended.  Section \ref{numvariablesgrows} presents a study of balanced perturbations as the number of variables grows.  It is in 
this section where the observation of Canteaut and Videau is extended.  Finally, as mentioned before, some examples of sporadic perturbations are presented in section \ref{sporadicpert}.

\section{Preliminaries}
\label{prelim}
Let $\mathbb{F}_2$ be the binary field, $\mathbb{F}_2^{\,n} = \{(x_1,\ldots, x_n) | x_i \in \mathbb{F}_2, i = 1, . . . , n\}$, and $F({\bf X}) = F(X_1, \ldots,X_n)$ be a polynomial in $n$ variables over 
$\mathbb{F}_2$. The exponential sum associated to $F$ over $\mathbb{F}_2$ is
\begin{equation}
S(F)=\sum_{x_1,\ldots,x_n\in \mathbb{F}_2} (-1)^{F(x_1,\ldots,x_n)}.
\end{equation}
A Boolean function $F$ is called balanced if $S(F) = 0$, i.e. the number of zeros and the number of ones are equal in the truth table of $F$. This property is important for some applications in 
cryptography.

Any symmetric Boolean function is a linear combination of elementary symmetric polynomials.   Let $\sigma_{n,k}$ be the elementary symmetric polynomial in $n$ variables of degree $k$. For example,
\begin{equation}
\sigma_{4,3} = X_1 X_2 X_3+X_1 X_4 X_3+X_2 X_4 X_3+X_1 X_2 X_4.
\end{equation}
Every symmetric Boolean function can be identified with an expression of the form
\begin{equation}
\label{genboolsym}
\sigma_{n,k_1}+\sigma_{n,k_2}+\cdots+\sigma_{n,k_s},
\end{equation}
where $1\leq k_1<k_2<\cdots<k_s$ are integers.  For the sake of simplicity, the notation $\sigma_{n,[k_1,\cdots,k_s]}$ is used to denote (\ref{genboolsym}).  For example,
\begin{eqnarray}
\sigma_{3,[2,1]}&=&\sigma_{3,2}+\sigma_{3,1}\\ \nonumber
&=& X_1 X_2+X_3 X_2+X_1 X_3+X_1+X_2+X_3.
\end{eqnarray}
It is not hard to show that if $1\leq k_1 < k_2 < \cdots < k_s$ are fixed integers, then
\begin{equation}
\label{maingen}
S(\sigma_{n,[k_1,k_2,\cdots,k_s]}) =\sum_{l=0}^n (-1)^{\binom{l}{k_1}+\binom{l}{k_2}+\cdots+\binom{l}{k_s}}\binom{n}{l}.
\end{equation}
\begin{remark}
Observe that the right hand side of (\ref{maingen}) makes sense for $n\geq 1$, while the left hand side exists for $n\geq k_s$. Throughout the rest of the article, $S(\sigma_{n,[k_1,k_2,\cdots,k_s]})$ 
should be interpreted as the expression on the right hand side, so it makes sense to talk about ``exponential sums" of symmetric Boolean functions with less variables than their degrees. 
\end{remark}

In \cite{cm1}, Castro and Medina used (\ref{maingen}) to study exponential sums of  symmetric polynomials from the point of view of integer sequences.  As part of their study, they showed that the sequence 
$\{S(\sigma_{n,[k_1,\cdots,k_s]})\}_{n\in \mathbb{N}}$ satisfies the homogeneous linear recurrence 
\begin{equation}
\label{mainrec}
x_n=\sum_{l=1}^{2^r-1}(-1)^{l-1}\binom{2^r}{l}x_{n-l},
\end{equation}
where $r=\lfloor\log_2(k_s)\rfloor+1$ (this result also follows from \cite[Th. 3.1, p. 248]{cai})  and used this result to compute the asymptotic behavior $S(\sigma_{n,[k_1,\cdots,k_s]})$ as $n\to \infty$.  
To be specific,
\begin{equation}
\label{asymplimit}
\lim_{n\to\infty}\frac{1}{2^n}S(\sigma_{n,[k_1,\cdots,k_s]}) = c_0(k_1,\cdots, k_s)
\end{equation}
where 
\begin{equation}
\label{theconstantscj}
c_0(k_1,\cdots,k_s) = \frac{1}{2^r} \sum_{l=0}^{2^r-1}(-1)^{\binom{l}{k_1}+\cdots+\binom{l}{k_s}}.
\end{equation}
They used this concept to show that a conjecture of Cusick, Li and St$\check{\mbox{a}}$nic$\check{\mbox{a}}$ is true asymptotically (this result was 
recently re-established in \cite{ggz}).    See \cite{cm1} for more details.

In the case of the elementary symmetric polynomial, the same authors were able to improve (\ref{mainrec}) and reduced the degree of the homogeneous linear recurrence with integer coefficients that its 
exponential sums satisfy.  They did this by finding the minimal homogeneous linear recurrence with integer coefficients that $\{S(\sigma_{n,k})\}$ satisfies.  To be specific, let $\epsilon(n)$ be defined as
\begin{equation}
\epsilon(n)=\left\{\begin{array}{cl}
    0, & \text{if }n \text{ is a power of 2,}  \\
    1, & \text{otherwise.}
       \end{array}\right.
\end{equation}
Then, the following result holds (see \cite{cm1}).
\begin{theorem}
\label{charpolyn}
Let $k$ be a natural number and $p_k(X)$ be the characteristic polynomial associated to the minimal linear recurrence with integer coefficients that $\{S(\sigma_{n,k})\}_{n\in\mathbb{N}}$ satisfies.  
Let $\bar{k}=2\lfloor k/2\rfloor+1$.  Express $\bar{k}$ as its $2$-adic expansion 
\begin{equation}
\bar{k}=1+2^{a_1}+2^{a_2}+\cdots+2^{a_s},
\end{equation}
where the last exponent is given by $a_s=\lfloor \log_2(\bar{k})\rfloor.$  Then,
\begin{equation}
\label{charpoly}
p_k(X)=(X-2)^{\epsilon(k)}\prod_{l=1}^s \Phi_{2^{a_l+1}}(X-1).
\end{equation}
In particular, the degree of the minimal linear recurrence that $\{S(\sigma_{n,k})\}_{n\in\mathbb{N}}$ satisfies is equal to $2\lfloor k/2 \rfloor + \epsilon(k)$.
\end{theorem}

In this article, perturbations of symmetric Boolean functions and their connection to solutions of (\ref{bisec}) over some bounded set of integers are considered.  Recall that $\sigma_{n,k}$ is the 
elementary symmetric polynomial of degree $k$ in the variables $X_1,\cdots, X_n$.  Suppose that $j<n$ and let $F({\bf X})$ be a binary polynomial in the variables $X_1, \cdots, X_j$ (the first $j$ variables 
in $X_1,\cdots, X_n$).  We are interested in exponential sums of polynomials of the form
\begin{equation}
\label{distor}
\sigma_{n,[k_1,\cdots,k_s]}+F({\bf X}),
\end{equation}
where $1\leq k_1<\cdots<k_s$.  Observe that perturbations of the form (\ref{distor}) are not necessarily symmetric.   

In \cite{cm2}, Castro and Medina showed that exponential sums of perturbations of the form (\ref{distor}) are related to exponential sums of symmetric Boolean functions via the following equation
\begin{equation}
\label{perturbationeq}
S(\sigma_{n,[k_1,\cdots,k_s]}+F({\bf X}))=\sum_{m=0}^j C_m(F)S\left(\sum_{i=0}^m\binom{m}{i}(\sigma_{n-j,[k_1-i,\cdots,k_s-i}])\right),
\end{equation}
where
\begin{equation}
C_m(F)=\sum_{{\bf x}\in \mathbb{F}_2 \text{ with }w_2({\bf x})=m}(-1)^{F({\bf x})},
\end{equation}
and $w_2({\bf x})$ represents the Hamming weight of ${\bf x}$, i.e. the number of entries of ${\bf x}$ that are one.
\begin{remark}
There are three things to observe about equation (\ref{perturbationeq}). First, it is clear that the value of $\binom{m}{i}$ that is inside the exponential sum can be taken mod 2, since only the parity 
matters.  Second, if $k_l-i <0$, then the term $\sigma_{n-j,k_l-i}$ does not exist and so it is not present in the equation.  Finally, in the case that $k_l-i=0$, the elementary polynomial $\sigma_{n-j,0}$ 
should be interpreted as 1.
\end{remark}
Equation (\ref{perturbationeq}) now implies that exponential sums of these type of perturbations also satisfy recurrence (\ref{mainrec}).  This is used in the next section when an identity between 
perturbations of two different symmetric Boolean polynomials is established.  It also serves as the link to equation (\ref{bisec}) over a bounded set of integers.

\section{Some perturbations identities}
\label{pertuident}
In this section we establish a very beautiful identity between perturbations of two different symmetric Boolean polynomials.  We start our discussion with a particular example.  The idea for doing this is to 
get an insight of what is behind this identity.  A proof for the general case will be provided later in this section once our intuition is solidified.

Consider the two polynomials $\sigma_{n,4}$ and $\sigma_{n,5}$ and their corresponding exponential sums:
\begin{equation}
\label{sums1}
S(\sigma_{n,4})=\sum_{l=0}^n (-1)^{\binom{l}{4}}\binom{n}{l} \,\,\, \text{ and } \,\,\, S(\sigma_{n,5})=\sum_{l=0}^n (-1)^{\binom{l}{5}}\binom{n}{l}.
\end{equation}
These sums seem similar, but they have very different behaviors.  For example, consider the expressions 
\begin{equation}
\label{seqnegone}
(-1)^{\binom{l}{4}}\,\,\, \text{ and } \,\,\, (-1)^{\binom{l}{5}},
\end{equation}
which are the coefficients of the binomial numbers in the sums (\ref{sums1}).  As $l$ ranges through the non-negative integers, both expressions in (\ref{seqnegone}) are periodic with period length 8.  In fact, their periods are given by
$$\begin{array}{|c|c|c|c|c|r|r|r|r|}
\hline
l & 0 & 1& 2& 3& 4& 5& 6& 7\\
\hline
(-1)^{\binom{l}{4}}& 1& 1& 1& 1& -1& -1& -1& -1 \\
(-1)^{\binom{l}{5}}& 1& 1& 1& 1& 1& -1& 1& -1\\
\hline
\end{array}$$
Observe that these periods differ in only two positions, but this difference has a big effect on the behavior of $S(\sigma_{n,4})$ and $S(\sigma_{n,5})$.  For instance, $S(\sigma_{n,4})=0$, 
i.e. $\sigma_{n,4}$ is balanced, whenever $n\equiv 7 \mod 8$.  In contrast, the exponential sum $S(\sigma_{n,5})$ is never zero.  Also, the exponential sum $S(\sigma_{n,4})$ assumes negative values for some 
values of $n$, but $S(\sigma_{n,5})$ is always positive. These claims are evident from the first few values of both sequences:
$$\begin{array}{|c|c|c|c|c|r|r|c|r|r|r|}
\hline
n & 1& 2& 3& 4& 5& 6& 7& 8& 9& 10\\
\hline
S(\sigma_{n,4})& 2& 4& 8& 14& 20& 20& 0& -68& -232& -560 \\
S(\sigma_{n,5})& 2& 4& 8& 16& 30& 52& 84& 128& 188& 280\\
\hline
\end{array}$$
Thus, it is clear that the sequences $\{S(\sigma_{n,4})\}$ and $\{S(\sigma_{n,5})\}$ have very different behavior.  However, both of them can be altered to make them, not just similar, but equal 
(up to a shift)!  The trick, in fact, is very simple: just add the linear polynomial $X_1$ to both symmetric polynomials, $\sigma_{n,4}$ and $\sigma_{n,5}$, to get
\begin{equation}
 S(\sigma_{n,4}+X_1)=S(\sigma_{n+1,5}+X_1).
\end{equation}
For example,
$$\begin{array}{|c|c|c|c|c|r|r|r|r|r|}
\hline
n & 2& 3& 4& 5& 6& 7& 8& 9& 10\\
\hline
S(\sigma_{n,4}+X_1)& 0& 0& 2& 8& 20& 40& 68& 96& 96 \\
S(\sigma_{n,5}+X_1)& 0& 0& 0& 2& 8& 20& 40& 68& 96\\
\hline
\end{array}$$
This trick not only works for $\sigma_{n,4}$ and $\sigma_{n,5}$, but also for $\sigma_{n,2k}$ and $\sigma_{n,2k+1}$ with $k$ a positive integer.  Explicitly,
\begin{equation}
\label{lingen}
 S(\sigma_{n,2k}+X_1)=S(\sigma_{n+1,2k+1}+X_1).
\end{equation}

We now begin our proof of identity (\ref{lingen}).  The bulk of the proof relies on the fact that sequences of the form $\{S(\sigma_{n,k})\}$ and $\{S(\sigma_{n,k}+F({\bf X}))\}$ satisfy linear recurrences
with integer coefficients.  To be more specific, the idea to establish identity (\ref{lingen}) is to show that both sequences satisfy the same recurrence.  Once this is done, then to show that 
identity (\ref{lingen}) holds, it is enough to show that both sequences have the same initial conditions.  We start our argument with the following result.

\begin{lemma}
\label{reclemma}
 Let $k>1$ and $m\geq1$ be fixed integers.  Consider the sequence
 \begin{equation}
 \label{seqdist}
\left\{S\left(\sum_{j=0}^m\binom{m}{j}\sigma_{n,k-j}\right)\right\}.  
 \end{equation}
 Then, sequence (\ref{seqdist}) satisfies the same homogeneous linear recurrence as $\{S(\sigma_{n,k})\}$.
\end{lemma}

\begin{proof}
 The proof is by induction on $m$.  Start with the identity
 \begin{equation}
  S(\sigma_{n,k})=S(\sigma_{n-1,k})+S(\sigma_{n-1,k}+\sigma_{n-1,k-1}).
 \end{equation}
This implies that 
\begin{equation}
 S(\sigma_{n,k}+\sigma_{n,k-1}) = S(\sigma_{n+1,k})-S(\sigma_{n,k}),
\end{equation}
and so it is clear that $\{S(\sigma_{n,k}+\sigma_{n,k-1})\}$ satisfies the same recurrence as $\{S(\sigma_{n,k})\}$.  Thus, the claim holds for $m=1$.  

Suppose now that the statement is true for all 
values of $m_1$ less than some $m>1$.  The identity
\begin{eqnarray}
  S(\sigma_{n,k})&=&\sum_{l=0}^m \binom{m}{l}S\left(\sum_{i=0}^l\binom{l}{i}\sigma_{n-m,k-i}\right)\\ \nonumber
  &=&\sum_{l=0}^{m-1} \binom{m}{l}S\left(\sum_{i=0}^l\binom{l}{i}\sigma_{n-m,k-i}\right)+S\left(\sum_{i=0}^m\binom{m}{i}\sigma_{n-m,k-i}\right),
\end{eqnarray}
implies
\begin{equation}
\label{indcsum}
S\left(\sum_{i=0}^m\binom{m}{i}\sigma_{n,k-i}\right)=S(\sigma_{n+m,k})-\sum_{l=0}^{m-1} \binom{m}{i}S\left(\sum_{i=0}^l\binom{l}{i}\sigma_{n,k-i}\right).
\end{equation}
By induction, each term of the sum on the right hand side of (\ref{indcsum}) satisfies the same recurrence as $\{S(\sigma_{n,k})\}$, therefore the claim is also true for $m$.  This concludes the proof.
\end{proof}

The next step is to show that the sequences $\{S(\sigma_{n,k})\}$ and $\{S(\sigma_{n,k}+F({\bf X}))\}$ satisfy the same linear recurrence with integer coefficients.

\begin{theorem}
\label{pertrec}
Let $k>1$ and $j$ be fixed integers and let $F({\bf X})$ be a binary polynomial in the variables $X_1,\cdots,X_j$. Suppose that $\bar{k}=2^{a_s}+\cdots+2^{a_1}+1$.  The sequence
\begin{equation}
\label{pertur}
\{S(\sigma_{n,k}+F({\bf X}))\}_{n\in \mathbb{N}} 
\end{equation}
satisfies the homogeneous linear recurrence whose characteristic polynomial is
\begin{equation}
f(X)=(X-2)^{\epsilon(k)}\prod_{l=1}^s \Phi_{2^{a_l+1}}(X-1).
\end{equation}
Moreover, if $F({\bf X})$ is balanced then (\ref{pertur}) satisfies the homogeneous linear recurrence with characteristic polynomial
\begin{equation}
\prod_{l=1}^s \Phi_{2^{a_l+1}}(X-1),
\end{equation}
which is of degree $\bar{k}-1=2^{a_s}+\cdots+2^{a_1}$.
\end{theorem}

\begin{proof}
Recall that
\begin{equation}
\label{lincombeq}
S(\sigma_{n,k}+F({\bf X}))=\sum_{m=0}^jC_m(F)S\left(\sum_{i=0}^m\binom{m}{i}\sigma_{n-j,k-i}\right),
\end{equation}
where
\begin{equation}
C_m(F) = \sum_{{\bf x}\in \mathbb{F}_2^j \text{ with }w_2({\bf x})=m} (-1)^{F({\bf x})},
\end{equation}
and $w_2({\bf x})$ is the Hamming weight of ${\bf x}$.  Thus, the first claim is a direct consequence of Lemma \ref{reclemma}.  For the second claim, the fact that $f(X)$ is the characteristic polynomial implies that
\begin{equation}
 S(\sigma_{n,k}+F({\bf X}))=d_0\cdot 2^n +\sum_{\lambda\neq 2, f(\lambda)=0} d_\lambda \cdot \lambda^n
\end{equation}
for some constants $d_0$ and $d_\lambda$'s.  In \cite{cm2}, Castro and Medina showed that 
\begin{equation}
 d_0 = c_0(k)\cdot\frac{S(F)}{2^j},
\end{equation}
where $c_0(k)$ is defined by (\ref{theconstantscj}).  The result follows.
\end{proof}

The above results are sufficient to show that for $k\geq 1$ an integer,  
\begin{equation}
S(\sigma_{n,2k}+X_1) = S(\sigma_{n+1,2k+1}+X_1) 
\end{equation}
for every positive integer $n$. Observe that since $F({\bf X})=X_1$ is a balanced polynomial, then Theorem \ref{pertrec} implies that the sequences $\{S(\sigma_{n,2k}+X_1)\}$ and 
$\{S(\sigma_{n+1,2k+1}+X_1)\}$ satisfy the same linear recurrence of order $2k$.  Therefore, to show that both sequences are equal, it is sufficient to show that their first $2k$ values coincide.

Let
\begin{eqnarray*}
 f(n,k)&=&S(\sigma_{n,2k}+X_1)\\
 &=& S(\sigma_{n-1,2k})-S(\sigma_{n-1,2k}+\sigma_{n-1,2k-1})\\
 &=& \sum_{j=0}^{n-1} (-1)^{\binom{j}{2k}}\left(1-(-1)^{\binom{j}{2k-1}}\right)\binom{n-1}{j}
\end{eqnarray*}
and
\begin{eqnarray*}
 g(n,k)&=&S(\sigma_{n+1,2k+1}+X_1)\\
 &=& S(\sigma_{n,2k+1})-S(\sigma_{n,2k+1}+\sigma_{n,2k})\\
 &=& \sum_{j=0}^{n} (-1)^{\binom{j}{2k+1}}\left(1-(-1)^{\binom{j}{2k}}\right)\binom{n}{j}.
\end{eqnarray*}
Note that
\begin{eqnarray*}
 f(1,k)=&0& = g(1,k),\\
 f(2,k)=&0& = g(2,k),\\
 &\vdots&\\
 f(2k-1,k)=&0& = g(2k-1,k),\\
 f(2k,k)=&2& = g(2k,k).
\end{eqnarray*}
In other words, $\{f(n,k)\}_{n=1}^\infty$ and $\{g(n,k)\}_{n=1}^\infty$ satisfy the same recurrence of order $2k$ with the same initial conditions.  Therefore,
\begin{eqnarray*}
 f(n,k)&=&g(n,k),\\
 S(\sigma_{n,2k}+X_1)&=&S(\sigma_{n+1,2k+1}+X_1)
\end{eqnarray*}
for every $n$ and $k$.  This discussion also implies the following identity of binomial sums.
\begin{corollary}
 Let $k$ and $n$ be positive integers.  Then,
 $$\sum_{l=0}^{n} (-1)^{\binom{l}{2k+1}}\left(1-(-1)^{\binom{l}{2k}}\right)\binom{n}{l}=\sum_{l=0}^{n-1} (-1)^{\binom{l}{2k}}\left(1-(-1)^{\binom{l}{2k-1}}\right)\binom{n-1}{l}.$$
\end{corollary}

Therefore, even though $\{S(\sigma_{n,2k})\}$ and $\{S(\sigma_{n,2k+1})\}$ are different sequences, they can be altered in such a way to produce an equality (up to a shift in the number of variables), i.e.
\begin{equation}
\label{lingen2k}
 S(\sigma_{n,2k}+X_1)=S(\sigma_{n+1,2k+1}+X_1).
\end{equation}
Equation (\ref{lingen2k}) leads to the following question: For which Boolean polynomials $F({\bf X})$ in $j$ variables ($j$ fixed) does the identity
\begin{equation}
\label{lin2kmoregen}
 S(\sigma_{n,2k}+F({\bf X}))=S(\sigma_{n+1,2k+1}+F{(\bf X}))
\end{equation}
hold?  Remarkably, the answer is for all balanced polynomials $F({\bf X})$.  

We prove this claim next.  However, its proof depends on the following classical result.

\begin{lemma}[Lucas' Theorem]
Let $n$ be a natural number with 2-adic expansion $n = 2^{a_1} + 2^{a_2} + \cdots + 2^{a_l}$. The binomial coefficient $\binom{n}{k}$ is odd if and only if $k$ is
either 0 or a sum of some of the $2^{a_i}$'s
\end{lemma}

\begin{theorem}
\label{samepert}
Suppose that $k\geq 1$ is an integer.  Consider a Boolean polynomial $F({\bf X})$ in $j$ (fixed) variables.  Then, $$S(\sigma_{n+j,2k}+F({\bf X})) = S(\sigma_{n+1+j,2k+1}+F({\bf X}))$$ for every positive integer $n$ if and only if $F({\bf X})$ is balanced.
\end{theorem}

\begin{proof}
The necessary part of the statement is not hard to establish.   To see this, suppose that $F({\bf X})$ is not balanced. Recall that 
\begin{eqnarray}
\label{expsumasbinom}
S(\sigma_{n+j,2k}+F({\bf X}))&=& \sum_{m=0}^jC_m(F)S\left(\sum_{i=0}^m\binom{m}{i}\sigma_{n,2k-i}\right)\\\nonumber
&=& \sum_{l=0}^n\left(\sum_{m=0}^j C_m(F)(-1)^{\sum_{i=0}^m \binom{m}{i}\binom{l}{2k-i}}\right)\binom{n}{l}.\\\nonumber
S(\sigma_{n+1+j,2k+1}+F({\bf X}))&=& \sum_{m=0}^jC_m(F)S\left(\sum_{i=0}^m\binom{m}{i}\sigma_{n+1,2k+1-i}\right)\\\nonumber
&=&\sum_{l=0}^{n+1}\left(\sum_{m=0}^j C_m(F)(-1)^{\sum_{i=0}^m \binom{m}{i}\binom{l}{2k-i}}\right)\binom{n+1}{l}.
\end{eqnarray}
Theorem \ref{pertrec} implies that $\{S(\sigma_{n+j,2k}+F({\bf X}))\}$ and $\{S(\sigma_{n+1+j,2k+1}+F({\bf X}))\}$ satisfy the same linear recurrence of order $2k+1$.  It is not hard to see 
that the first initial condition of $\{S(\sigma_{n,2k}+F({\bf X}))\}$ is $S(F)$, while the first initial condition of $\{S(\sigma_{n+1,2k+1}+F({\bf X}))\}$ is $2S(F)$.  Since $S(F)\neq 0$, then 
$\{S(\sigma_{n+j,2k}+F({\bf X}))\}$ and $\{S(\sigma_{n+1+j,2k+1}+F({\bf X}))\}$ are different sequences.

The sufficient part can also be argued from the point of view of recurrences, but the proof becomes cumbersome as soon as $j\geq 5$.  Therefore, we must find an alternative proof.   Suppose that $F({\bf X})$ is balanced.
For simplicity, let $C_m=C_m(F)$ and remember that $S(F)=C_0+C_1+\cdots+C_j$.   The balancedness of $F({\bf X})$ implies that
\begin{equation}
C_0+C_1+\cdots+C_j=0.
\end{equation}
We now simplify the formulas in (\ref{expsumasbinom}).  For this, iterate the recursive formula
\begin{equation}
\label{recur}
\binom{n}{l}=\binom{n-1}{l}+\binom{n-1}{l-1}
\end{equation}
to obtain the identity
\begin{equation}
\sum_{i=0}^m\binom{m}{i}\binom{l}{k-i}=\binom{l+m}{k}.
\end{equation}
Thus,
\begin{eqnarray}
\label{expsumasbinomsimp}
S(\sigma_{n+j,2k}+F({\bf X}))&=& \sum_{l=0}^n\left(\sum_{m=0}^j C_m(-1)^{\binom{l+m}{2k}}\right)\binom{n}{l},\\\nonumber
S(\sigma_{n+1+j,2k+1}+F({\bf X}))&=& \sum_{l=0}^n\left(\sum_{m=0}^j C_m(-1)^{\binom{l+1+m}{2k+1}}\right)\binom{n+1}{l}.
\end{eqnarray}
Observe that $S(\sigma_{n+1+j,2k+1}+F({\bf X}))$ can be re-written as
\begin{eqnarray*}
S(\sigma_{n+1+j,2k+1}+F({\bf X}))&=& \sum_{l=0}^n\left(\sum_{m=0}^j C_m(-1)^{\binom{l+m}{2k+1}}\right)\left[\binom{n}{l}+\binom{n}{l-1}\right]\\
&=& \sum_{l=0}^n\left(\sum_{m=0}^j C_m\left[(-1)^{\binom{l+m}{2k+1}}+(-1)^{\binom{l+1+m}{2k+1}}\right]\right)\binom{n}{l}.
\end{eqnarray*}

The next step is to use the assumption of balancedness of $F({\bf X})$.  Since 
\begin{equation}
C_0=-C_1-C_2-\cdots-C_j,
\end{equation}
then
\begin{eqnarray}
S(\sigma_{n+j,2k}+F({\bf X}))&=& \sum_{m=0}^j C_m\sum_{l=0}^n(-1)^{\binom{l+m}{2k}}\binom{n}{l}\\\nonumber
&=& \sum_{m=1}^j C_m\sum_{l=0}^n\left[(-1)^{\binom{l+m}{2k}}-(-1)^{\binom{l}{2k}}\right]\binom{n}{l}\\\nonumber
\end{eqnarray}
and
\begin{eqnarray}\nonumber
S(\sigma_{n+1+j,2k+1}+F({\bf X}))&=& \sum_{m=0}^j C_m\sum_{l=0}^n\left[(-1)^{\binom{l+m}{2k+1}}+(-1)^{\binom{l+1+m}{2k+1}}\right]\binom{n}{l}\\
&=& \sum_{m=1}^j C_m\sum_{l=0}^n\left[(-1)^{\binom{l+m}{2k+1}}+(-1)^{\binom{l+1+m}{2k+1}}-(-1)^{\binom{l}{2k+1}}-(-1)^{\binom{l+1}{2k+1}}\right]\binom{n}{l}.
\end{eqnarray}
One consequence of Lucas' Theorem is
\begin{equation}
(-1)^{\binom{l+m}{2k}}-(-1)^{\binom{l}{2k}} = (-1)^{\binom{l+m}{2k+1}}+(-1)^{\binom{l+1+m}{2k+1}}-(-1)^{\binom{l}{2k+1}}-(-1)^{\binom{l+1}{2k+1}}
\end{equation}
for all positive integers $k$ and all non negative integers $l$ and $m$.  This concludes the proof.
\end{proof}

\begin{example}
Consider the rotation $$R({\bf X})=X_1X_2+X_2X_3+X_3X_4+X_4X_5+X_5X_1.$$ Observe that $R({\bf X})$ is balanced.  Theorem \ref{samepert}  implies that $S(\sigma_{n,2k}+R({\bf X})) = S(\sigma_{n+1,2k+1}+
R({\bf X}))$ for every positive integer $n$ and $k$.  Indeed, let $2k=10$, then the first few values of the sequence $\{S(\sigma_{n,10}+R({\bf X}))\}$ (starting from $n=10$) are
$$2, 24, 136, 528, 1612, 4144, 9336, 18928, 35220, 61104, 100064,\cdots$$
while the first few values of $\{S(\sigma_{n,11}+R({\bf X}))\}$ (starting from $n=10$) are
$$0, 2, 24, 136, 528, 1612, 4144, 9336, 18928, 35220, 61104,\cdots.$$
\end{example}

\begin{example}
Consider the Boolean polynomial 
\begin{eqnarray*}
F({\bf X})&=&X_1 X_5+X_4 X_5+X_9 X_5+X_{12} X_5+X_3 X_6+X_2 X_7+X_1 X_8+X_1 X_9\\
&&+X_4 X_9+X_8 X_9+X_3 X_{10}+X_7 X_{10}+X_2 X_{11}+X_6 X_{11}+X_1 X_{12}. 
\end{eqnarray*}
Theorem \ref{pertrec} implies that the sequences $\{S(\sigma_{n,2}+F({\bf X}))\}$ and $\{S(\sigma_{n+1,3}+F({\bf X}))\}$ satisfy the same linear recurrence of order 3.   Now observe that
\begin{eqnarray}
S(\sigma_{13,2}+F({\bf X}))&=&S(\sigma_{14,3}+F({\bf X}))=0\\\nonumber
S(\sigma_{14,2}+F({\bf X}))&=&S(\sigma_{15,3}+F({\bf X}))=-256\\\nonumber
S(\sigma_{15,2}+F({\bf X}))&=&S(\sigma_{16,3}+F({\bf X}))=-512.
\end{eqnarray}
Thus, $\{S(\sigma_{n,2}+F({\bf X}))\}=\{S(\sigma_{n+1,3}+F({\bf X}))\}$ holds for every $n$. In view of Theorem \ref{samepert}, it must be that $F({\bf X})$ is balanced.  Indeed, $S(F)=0$.
\end{example}

It is often the case in mathematics that results can be generalized by closer inspection of their proofs.  This is the case for Theorem \ref{samepert}. A crucial element of its proof is the identity
\begin{equation}
\label{neg1powers}
(-1)^{\binom{l+m}{2k}}-(-1)^{\binom{l}{2k}} = (-1)^{\binom{l+m}{2k+1}}+(-1)^{\binom{l+1+m}{2k+1}}-(-1)^{\binom{l}{2k+1}}-(-1)^{\binom{l+1}{2k+1}},
\end{equation}
which is true for all positive integers $k$ and all non negative integers $l$ and $m$.  Now, iteration of (\ref{recur}) leads to the equation
\begin{equation}
 \binom{l+m}{2k}=\sum_{i=0}^t \binom{t}{i}\binom{l+m-t}{2k-i}
\end{equation}
and other similar ones for 
$$\binom{l}{2k},\,\,\, \binom{l+m}{2k+1},\,\,\, \binom{l+m+1}{2k+1},\,\,\, \binom{l}{2k+1}, \text{ and } \binom{l+1}{2k+1}.$$
This and equation (\ref{neg1powers}) (shifted by $t$) imply that 
\begin{eqnarray}\nonumber
\label{neg1powersgen}
(-1)^{\sum_{i=0}^t \binom{t}{i}\binom{l+m}{2k-i}}-(-1)^{\sum_{i=0}^t \binom{t}{i}\binom{l}{2k-i}}&=&(-1)^{\sum_{i=0}^t \binom{t}{i}\binom{l+m}{2k+1-i}}+(-1)^{\sum_{i=0}^t \binom{t}{i}\binom{l+m+1}{2k+1-i}}\\
&& -(-1)^{\sum_{i=0}^t \binom{t}{i}\binom{l}{2k+1-i}}-(-1)^{\sum_{i=0}^t \binom{t}{i}\binom{l+1}{2k+1-i}}
\end{eqnarray}
is true for all positive integers $k$ and all non negative integers $l$ and $m$.

Suppose that $F({\bf X})$ is a balanced Boolean polynomial in the variables $X_1,\cdots, X_j$ with $j$ fixed.  Following the proof of Theorem \ref{samepert} and using equation (\ref{neg1powersgen}) lead to 
the equation
\begin{equation}
 S\left(\left[\sum_{i=0}^t\binom{t}{i}\sigma_{n+j,2k-i}\right]+F({\bf X})\right) = S\left(\left[\sum_{i=0}^t\binom{t}{i}\sigma_{n+1+j,2k+1-i}\right]+F({\bf X})\right),
\end{equation}
which is true for all positive integers $k$ and all non-negative integers $t$.  This leads to the following result.
\begin{theorem}
\label{samepertgen}
 Suppose that $k$ and $t$ are integers with $k$ positive and $t$ non-negative. Consider a Boolean polynomial $F({\bf X})$ in $j$ (fixed) variables.  Then, 
 $$ S\left(\left[\sum_{i=0}^t\binom{t}{i}\sigma_{n+j,2k-i}\right]+F({\bf X})\right) = S\left(\left[\sum_{i=0}^t\binom{t}{i}\sigma_{n+1+j,2k+1-i}\right]+F({\bf X})\right)$$ 
 for every positive integer $n$ if and only if $F({\bf X})$ is balanced.
\end{theorem}
\begin{proof}
Lemma \ref{reclemma} and Theorem \ref{pertrec} can be easily extended to perturbations of the form $\sigma_{n,[k_1,\cdots,k_s]}+F({\bf X})$.  Once this is done, the necessary part of the statement follows as in 
Theorem \ref{samepert}.  The sufficient part follows from the discussion above.
\end{proof}

\begin{example}
 Suppose that $F({\bf X})$ is a balanced Boolean polynomial in the variables $X_1,\cdots, X_j$ with $j$ fixed.
 Theorem \ref{samepertgen} implies that the following equations
\begin{align*}
  S(\sigma_{n+j,6}+F({\bf X}))&=S(\sigma_{n+j+1,7}+F({\bf X}))\\
  S(\sigma_{n+j,[6,5]}+F({\bf X}))&=S(\sigma_{n+j+1,[7,6]}+F({\bf X}))\\
  S(\sigma_{n+j,[6,4]}+F({\bf X}))&=S(\sigma_{n+j+1,[7,5]}+F({\bf X}))\\
  S(\sigma_{n+j,[6,5,4,3]}+F({\bf X}))&=S(\sigma_{n+j+1,[7,6,5,4]}+F({\bf X}))\\
  S(\sigma_{n+j,[6,2]}+F({\bf X}))&=S(\sigma_{n+j+1,[7,3]}+F({\bf X}))\\
  S(\sigma_{n+j,[6,5,2,1]}+F({\bf X}))&=S(\sigma_{n+j+1,[7,6,3,2]}+F({\bf X}))\\
  S(\sigma_{n+j,[6,4,2,0]}+F({\bf X}))&=S(\sigma_{n+j+1,[7,5,3,1]}+F({\bf X}))\\
  S(\sigma_{n+j,[6,5,4,3,2,1]}+F({\bf X}))&=S(\sigma_{n+j+1,[7,6,5,4,3,2,1]}+F({\bf X})),
\end{align*}
are true for every natural number $n$.
\end{example}

\section{Diophantine equations with binomial coefficients}
\label{diophsec}
In this section we are interested in Diophantine equations of the form
\begin{equation}
\label{gensec}
 \sum_{l=0}^n \binom{n}{l} x_l=0,
\end{equation}
where $x_l$ belongs to some fixed bounded subset $\Gamma$ of $\mathbb{Z}$.   Exponential sums of symmetric Boolean functions are naturally connected to this problem.

Recall that exponential sums of symmetric Boolean functions can be expressed in terms of binomial sums.  To be specific, if $1\leq k_1<\cdots<k_s$ are integers, then
\begin{equation}
S(\sigma_{n,[k_1,\cdots,k_s]}) = \sum_{l=0}^n (-1)^{\binom{l}{k_1}+\cdots+\binom{l}{k_s}}\binom{n}{l}.
\end{equation}
Therefore, every time we find a balanced symmetric Boolean function, we also find a solution to (\ref{gensec}) where the $x_l$'s belong to the set $\Gamma=\{-1,1\}$.  The converse is also true, that is, if
we find a solution to  (\ref{gensec}) over $\Gamma=\{-1,1\}$, then we also find a balanced symmetric Boolean function.  For example, consider the equation
\begin{equation}
\label{sporadicfirstex}
 \binom{8}{0}-\binom{8}{1}-\binom{8}{2}-\binom{8}{3}+\binom{8}{4}+\binom{8}{5}-\binom{8}{6}-\binom{8}{7}+\binom{8}{8}=0.
\end{equation}
The corresponding balanced symmetric Boolean function is $\sigma_{8,[1,2,3,5,7]}$.  When the set considered is 
$\Gamma=\{-1,1\}$, then any solution to (\ref{gensec}) is said to give a bisection of the binomial coefficients $\binom{n}{l}$, $0\leq l \leq n$. Observe that such solution provides us with two disjoints 
sets $A$ and $A'$ such that $A\cup A'=\{0,1,2,\cdots, n\}$ and
$$\sum_{l\in A}\binom{n}{l}=\sum_{l\in A'}\binom{n}{l}=2^{n-1}.$$
As mentioned in the introduction, the problem of bisecting binomial coefficients was first discussed by Mitchell \cite{mitchell}.

Continue with equation (\ref{gensec}) over $\Gamma=\{-1,1\}$.  If $n$ is even, then $\delta_l = \pm(-1)^l$, for $l=0,1,\cdots, n$, are two solutions to (\ref{gensec}).  On the other hand,
if $n$ is odd, then the symmetry of the binomial coefficients  implies that $(\delta_0,\cdots, \delta_{(n-1)/2},-\delta_{(n-1)/2},\cdots,-\delta_0)$  are $2^{(n+1)/2}$ solutions to (\ref{gensec}).
These are called trivial solutions.  A balanced symmetric Boolean function in $n$ variables which corresponds to one of the trivial solutions of (\ref{gensec}) over $\Gamma=\{-1,1\}$ is said to be a {\it trivially balanced function}.  
Computations suggest that a majority of the balanced symmetric Boolean functions are trivially balanced, thus it is of great interest to find non-trivially balanced symmetric Boolean functions.  In 
the literature, these functions are called {\it sporadic} balanced symmetric Boolean functions.  For example, the relation (\ref{sporadicfirstex}) is not trivial, therefore $\sigma_{8,[1,2,3,5,7]}$ is a
sporadic balanced symmetric Boolean function.  See \cite{cusick3, cls, jeff} for more information.  In \cite{sarkarmaitra}, Sarkar and Maitra showed that there is an infinite amount of sporadic balanced symmetric Boolean functions.

The above discussion shows the link between balancedness of symmetric Boolean functions and solutions to (\ref{gensec}) over the set $\Gamma=\{-1,1\}$.  We now turn our attention to balancedness of 
perturbations of the form $\sigma_{n,[k_1,\cdots,k_s]}+F({\bf X})$, where $1\leq k_1<\cdots<k_s$ are integers and $F({\bf X})$ is a non-zero Boolean polynomial in $j$ (fixed) variables, and its connection to equation 
(\ref{gensec}).  For simplicity of the writing, consider the case $\sigma_{n,k}+F({\bf X})$. Recall that
\begin{equation}
S(\sigma_{n,k}+F({\bf X}))=\sum_{m=0}^jC_m(F)S\left(\sum_{i=0}^m\binom{m}{i}\sigma_{n-j,k-i}\right),
\end{equation}
where $C_m(F)$, for $m=0,1,\cdots, j$, are given by
\begin{equation}
C_m(F) = \sum_{{\bf x}\in \mathbb{F}_2^j \text{ with }w_2({\bf x})=m} (-1)^{F({\bf x})}.
\end{equation}
Expressing exponential sums of symmetric Boolean functions as binomial sums leads to 
\begin{equation}
S(\sigma_{n+j,k}+F({\bf X}))=\sum_{l=0}^n\left(\sum_{m=0}^j C_m(F)(-1)^{\sum_{i=0}^m \binom{m}{i}\binom{l}{k-i}}\right)\binom{n}{l}.
\end{equation}
Observe that
\begin{eqnarray}
\left|\sum_{m=0}^j C_m(F)(-1)^{\sum_{i=0}^m \binom{m}{i}\binom{l}{k-i}}\right|&\leq& \sum_{m=0}^j |C_m(F)| \\\nonumber
&\leq& \sum_{m=0}^j \binom{j}{m} = 2^j.
\end{eqnarray}
Also,
\begin{equation}
\sum_{m=0}^j C_m(F)(-1)^{\sum_{i=0}^m \binom{m}{i}\binom{l}{k-i}}\equiv \sum_{m=0}^j C_m(F)= S(F)\equiv 0 \mod 2.
\end{equation}
Therefore, balancedness of a perturbation of the form $\sigma_{n+j,k}+F({\bf X})$ is connected to solutions of (\ref{gensec}) over the set 
$$\Gamma_j^{(e)}=\{x\in 2\mathbb{Z}\,:\, |x|\leq 2^j\}=\{0,\pm2, \pm4,\pm 6,\cdots, \pm 2^j\}.$$  
Any solution to (\ref{gensec}) over $\Gamma_j^{(e)}$ can be divided by 2 to produce a solution of (\ref{gensec}) over the set
$$\Gamma_j=\{x\in \mathbb{Z}\,:\, |x|\leq 2^{j-1}\}=\{0,\pm1, \pm2,\pm 3,\cdots, \pm 2^{j-1}\}.$$
The opposite is clearly true, that is, any solution to (\ref{gensec}) over $\Gamma_j$ can be multiplied by 2 to produce a solution over $\Gamma^{(e)}_j$.  Therefore, in this study, most of the 
results are written in the language of the set $\Gamma_j$.  
\begin{remark}
The same conclusion can be reached from a perturbation of the form $\sigma_{n+j,[k_1,\cdots,k_s]}+F({\bf X})$.  Also, observe that if $F({\bf X})$ is the zero polynomial, then we are back at the problem of
bisecting binomial coefficients and therefore the corresponding set is $\Gamma=\{-1,1\}$.
\end{remark}

As in the case for bisections of binomial coefficients, we can define {\em trivial solutions} to (\ref{gensec}) over $\Gamma_j$.
 If $n$ is odd, then the symmetry of the binomial coefficients implies that 
\begin{equation}
\label{trivial1}
(\delta_0,\cdots, \delta_{(n-1)/2},-\delta_{(n-1)/2},\cdots,-\delta_0), 
\end{equation}
with $\delta_i \in \Gamma_j$, are $(2^j+1)^{\frac{n+1}{2}}$ solutions to (\ref{gensec}).  If $n$ is even, then 
\begin{equation}
\label{trivial2}
\delta_l = (-1)^l m, \text{ for } l=0,1,\cdots, n, \text{ and }m\in \Gamma_j 
\end{equation}
are $2^{j}+1$ solutions to (\ref{gensec}) over $\Gamma_j$.  Also, the symmetry of the binomial coefficients implies that 
\begin{equation}
\label{trivial3}
(\delta_0,\cdots, \delta_{n/2-1},0,-\delta_{n/2-1},\cdots,-\delta_0), 
\end{equation}
with $\delta_i \in \Gamma_j$ are $(2^j+1)^{\frac{n}{2}}$ solutions to (\ref{gensec}) over $\Gamma_j$ (observe that the trivial solution $(0,0,\cdots,0)$ is of the forms (\ref{trivial2}) and (\ref{trivial3})).  
Solutions of the form (\ref{trivial1}) for $n$ odd or of the forms (\ref{trivial2}) or (\ref{trivial3}) for $n$ even are called trivial solutions to (\ref{gensec}) over $\Gamma_j$.  

There are other solutions 
which at first sight do not seem to be trivial, for example, 
\begin{equation}
\label{notintrivialform}
 -2\binom{12}{3}+2\binom{12}{4}+\binom{12}{6}-2\binom{12}{7}+2\binom{12}{10}-2\binom{12}{11}+2\binom{12}{12}=0.
\end{equation}
However, using the symmetry of binomial numbers, equation (\ref{notintrivialform}) can be re-written as
\begin{equation}
\label{trivialform}
 \sum_{i=0}^{12} (-1)^i \binom{12}{i}=0,
\end{equation}
which is of the form (\ref{trivial2}).  This invites us to define equivalence of solutions.  We say that two solutions $(\delta_0^{(1)},\delta_1^{(1)},\cdots, \delta_n^{(1)})$ and 
$(\delta_0^{(2)},\delta_1^{(2)},\cdots, \delta_n^{(2)})$ are equivalent, and write $(\delta_0^{(1)},\delta_1^{(1)},\cdots, \delta_n^{(1)})\sim(\delta_0^{(2)},\delta_1^{(2)},\cdots, \delta_n^{(2)})$,
if
\begin{enumerate}
 \item both are non-zero solutions and
 $$\frac{1}{g_2}(\delta_0^{(2)},\delta_1^{(2)},\cdots, \delta_n^{(2)})=\pm \frac{1}{g_1}(\delta_0^{(1)},\delta_1^{(1)},\cdots, \delta_n^{(1)}),$$
 where $g_i=\gcd(\delta_0^{(i)},\delta_1^{(i)},\cdots, \delta_n^{(i)})$, or
 \item one solution can be obtained from the other by using the symmetry of the binomial numbers, as it is the case of (\ref{notintrivialform}) and (\ref{trivialform}), or
 \item one solution can be obtained from the other by combining the previous two cases.
\end{enumerate}
Because of this, we now say that solutions of the form (\ref{trivial1}), (\ref{trivial2}) or (\ref{trivial3}) are written in {\em trivial form} (or just that they are {\em trivial form solutions}) and 
extend the definition of {\em trivial solution} to any solution that is equivalent to one of the trivial form solutions.  For example, (\ref{notintrivialform}) is a trivial solution.  For 
${\boldsymbol{\delta}}=(\delta_0,\delta_1,\cdots,\delta_n) \in \Gamma_j$, define
$$[\delta_0,\delta_1,\cdots,\delta_n]=\{(\delta'_0,\delta'_1,\cdots,\delta'_n) \in \Gamma_j\,|\, (\delta'_0,\delta'_1,\cdots,\delta'_n)\sim (\delta_0,\delta_1,\cdots,\delta_n) \},$$
that is, $[\delta_0,\delta_1,\cdots,\delta_n]$ is the equivalence class of ${\boldsymbol{\delta}}$ under $\sim$.  Observe that if $n$ is odd, then every trivial form solution, and therefore every trivial 
solution, belongs to the class $[0,0,\cdots,0]$.  If $n$ is even, then every trivial solution belongs to either $[0,0,\cdots, 0]$ or $[1,-1,1,-1,\cdots,-1,1]$.

As expected, the number of solutions to (\ref{gensec}) over $\Gamma_j$ grows exponentially in $n$.  This can already be seen in the number of trivial form solutions, since there are $(2^j+1)^{\frac{n+1}{2}}$
such solutions for $n$ odd and $(2^j+1)^{\frac{n}{2}}+2^{j}$ for $n$ even.  In fact, let $\Omega(n,j) = \{{\boldsymbol{\delta}} \in \Gamma_j^n\,|\, {\boldsymbol{\delta}} \text{ is a solution to }(\ref{gensec})\}$
and define $\gamma_j(n):=|\Omega(n,j)|$, that is, the number of solutions to (\ref{gensec}) that lie in $\Gamma_j$.  Table \ref{numberofsols} contains $\gamma_j(n)$ for various $n$'s and $j$'s.  
\begin{table}[h!]
\caption{Number of solutions to (\ref{gensec}) that lie in $\Gamma_j$.}
\centering
\begin{tabular}{|c|c|c|c|c|c|c|c|c|c|c|}
\hline
$n$ & 1 & 2& 3& 4& 5& 6& 7& 8& 9& 10\\
\hline
$\gamma_1(n)$ & 3& 5& 9& 15& 39& 45& 129& 149& 243& 369\\
$\gamma_2(n)$ & 5& 13& 41& 103& 275& 685& 2525& 5221& 13897& 32717\\
$\gamma_3(n)$ & 9& 41& 219& 1033& 5181& 23035& 121921& *& *& * \\
$\gamma_4(n)$ & 17& 145& 1469& 12969& 120521& *& *& *& *& * \\
$\gamma_5(n)$ & 33& 545& 10659& 183477&*& *& *& *& *& *\\
$\gamma_6(n)$ & 65& 2113&81421&*&*& *& *& *& *& *\\
$\gamma_7(n)$ & 129& 8321&636099&*&*& *& *& *& *& * \\
\hline
\end{tabular}
\label{numberofsols}
\end{table}
\begin{remark}
The reader is invited to read the very interesting paper of Iona\c{s}cu, Martinsen and St$\check{\mbox{a}}$nic$\check{\mbox{a}}$ about bisecting binomial coefficients \cite{ims}.  As part of their work, 
they consider the number of nontrivial bisections.  They found many previously unknown infinite classes of integers which admit nontrivial bisections, and a class with only trivial bisections. They
also found several bounds for the number of nontrivial bisections and compute the exact number of such bisections for $n \leq 51$.  Some of these results may be extended to 
perturbations of symmetric Boolean functions.  For example, following similar techniques form \cite{ims} we obtained an integral representation for $\gamma_j(n)$.  Let $\mathbb{V}_j=[0,2^{j-1}]\cap 
\mathbb{Z}$. The notation ${\bf x}=(x_0,x_1,\cdots, x_n)$ is used to represent a vector
${\bf x}\in \mathbb{V}_j^{n+1}$. For ${\bf x}\in \mathbb{V}_j^{n+1}$, let $w({\bf x})$ represents the number of non-zero entries of ${\bf x}$.  The number $\gamma_j(n)$ is given by the following integral
\begin{equation}
\label{integralcounts}
\gamma_j(n)=\sum_{{\bf x} \in \mathbb{V}_j^{n+1}}2^{w({\bf x})}\int_0^1 \prod_{i=0}^{n}\cos\left(\pi x_i \binom{n}{i} s\right)\,ds.
\end{equation}
\end{remark}

Many of the solutions that are counted in Table \ref{numberofsols} are equivalent to some others. Therefore, the amount of ``meaningful" solutions should be significantly 
smaller than the numbers that are presented in Table \ref{numberofsols}.  Define $\omega_j(n)$ to be the number of different equivalence classes on $\Omega(n,j)$ under $\sim$, that is, the cardinality of 
the quotient set $\Omega(n,j)/\sim$.  For example,
\begin{equation}
 \Omega(4,2)/\sim\,\, = \{[0,0, 0,0,0], [0,1,-2,2,0], [2,-2, 1,0,0], [2, 1, -1,0,0], [2, -1, 0,0,2]\},
\end{equation}
and therefore $\omega_2(4) = 5$.  Table \ref{numberofequivsols} contains $\omega_j(n)$ for various $n$'s and $j$'s.  The order of magnitude of these numbers is smaller than the ones on Table 
\ref{numberofsols}, as expected, however, it is apparent that there are many meaningful solutions that are not trivial.

\begin{table}[ht!]
\caption{Number of different equivalence classes on $\Omega(n,j)$ under $\sim$.}
\centering
\begin{tabular}{|c|c|c|c|c|c|c|c|c|c|c|}
\hline
$n$ & 1 & 2& 3& 4& 5& 6& 7& 8& 9& 10\\
\hline
$\omega_1(n)$ & 1& 2& 1& 3& 2& 3& 3& 7& 1& 5\\
$\omega_2(n)$ & 1& 2& 2& 5& 2& 13& 7& 36& 26& 71\\
$\omega_3(n)$ & 1& 2& 2& 13& 10& 72& 77& 389& 274& 1681 \\
$\omega_4(n)$ & 1& 2& 2& 45& 37& 504& 443& 5076& 4336& * \\
$\omega_5(n)$ & 1& 2& 2& 161& 127& 3811& 3119& *& *& *\\
$\omega_6(n)$ & 1& 2& 2& 649 & 481& 29742& *& *& *& *\\
$\omega_7(n)$ & 1& 2& 2& 2521&2005& *& *& *& *& * \\
\hline
\end{tabular}
\label{numberofequivsols}
\end{table}

A balanced perturbation 
of the form $\sigma_{n+j,[k_1,\cdots,k_s]}+F({\bf X})$ is said to be {\em trivially balanced} if it corresponds to one of the trivial solutions of (\ref{gensec}) over $\Gamma_j$.  Analogous to the case for bisections of binomial 
coefficients, a non-trivially balanced perturbation of the form $\sigma_{n+j,[k_1,\cdots,k_s]}+F({\bf X})$ is called {\it sporadic} balanced Boolean perturbation.
As it is the case of balanced symmetric Boolean functions, many balanced perturbations seem to be trivially balanced.  For example, consider $\sigma_{n,k}+X_1$, which is the simplest perturbation to 
an elementary symmetric Boolean polynomial. We have the following result.
\begin{theorem}
\label{trivialbalanced1}
 Let $k$ be a positive integer.  Let $r=\lfloor \log_2 (k)\rfloor+1$. The perturbation $\sigma_{n,k}+X_1$ is (trivially) balanced for $n=2^{r}m+k-1$ where $m$ is any natural number.
\end{theorem}

\begin{proof}
Recall that 
\begin{eqnarray}
\label{easyper}
 S(\sigma_{n,k}+X_1)&=&S(\sigma_{n-1,k})-S(\sigma_{n-1,[k,k-1]})\\\nonumber
 &=& \sum_{l=0}^{n-1}(-1)^{\binom{l}{k}}\binom{n-1}{l}-\sum_{l=0}^{n-1}(-1)^{\binom{l}{k}+\binom{l}{k-1}}\binom{n-1}{l}.
\end{eqnarray}
Let $N(k)$ be the set of integers $l$ such that $\binom{l}{k}$ is odd.  Similarly, define $N(k,k-1)$ to be the set of integers $l$ such that $\binom{l}{k}+\binom{l}{k-1}$ is odd. Lucas' Theorem 
implies that $l\in N(k)$ if and only if $l-1\in N(k,k-1)$.  Also, the choice $n=2^rm+k-1$ implies that if $0\leq l \leq n$, then $l\in N(k)$ if and only if $n-l \in N(k)$.  

Suppose that $l_0\in N(k)$, that is, that the coefficient of 
\begin{equation}
\binom{n-1}{l_0} 
\end{equation}
in the first sum of (\ref{easyper}) is $-1$.  By the choice of $n$ one has $n-l_0\in N(k)$.  But if this is true, then $n-l_0-1\in N(k,k-1)$ and therefore the coefficient of 
\begin{equation}
 \binom{n-1}{n-l_0-1}=\binom{n-1}{n-1-(n-l_0-1)}=\binom{n-1}{l_0}
\end{equation}
in the second sum of (\ref{easyper}) is also $-1$.  The converse is also true.  If $l_1 \in N(k,k-1)$, then the coefficient of 
\begin{equation}
 \binom{n-1}{l_1}
\end{equation}
in the second sum of (\ref{easyper}) is $-1$.  But then, $l_1+1\in N(k)$ and by the choice of $n$ one has that $n-(l_1+1)=n-1-l_1 \in N(k)$.  Thus, the coefficient of 
\begin{equation}
 \binom{n-1}{n-1-l_1} = \binom{n-1}{l_1}
\end{equation}
in the first sum of (\ref{easyper}) is also $-1$.  We conclude that $\sigma_{n,k}+X_1$ is trivially balanced.
\end{proof}

\begin{example}
 Consider the perturbation $\sigma_{20,5}+X_1$.  Theorem \ref{trivialbalanced1} implies that this perturbation is trivially balanced.  Explicitly,
 \begin{eqnarray}
 \nonumber
  S(\sigma_{20,5}+X_1)&=&\sum_{l=0}^{19}\left((-1)^{\binom{l}{5}}-(-1)^{\binom{l}{5}+\binom{l}{4}}\right)\binom{19}{l}\\
  &=& 2\binom{19}{4} -2 \binom{19}{5}+ 2\binom{19}{6}-2\binom{19}{7}+2\binom{19}{12}-2 \binom{19}{13}+2\binom{19}{14}-2\binom{19}{15}\\\nonumber
  &=& 0 \,\,\, (\text{trivially}).
 \end{eqnarray}
\end{example}

We had tried to find sporadic balanced perturbations of the form $\sigma_{n,k}+X_1$, but our attempts failed.  This led us to believe that a conjecture similar to the one presented by Cusick, Li and 
St$\check{\mbox{a}}$nic$\check{\mbox{a}}$ in \cite{cls} for the case of elementary symmetric Boolean functions holds (also see \cite{cls2}).  To be specific, we conjecture the following:

\begin{conjecture}
No perturbation of the form $\sigma_{n,k}+X_1$ is balanced except for the trivial cases, i.e. when $n=2^{r}m+k-1$ where $r=\lfloor \log_2(k_s)\rfloor+1$ and $m$ is a positive integer.
\end{conjecture}

Similar techniques can be used to show that other infinite families are trivially balanced.  However, we will not show them here, except for the next example (for which we omit the proof).

\begin{theorem}
The perturbation
 \begin{equation}
 \sigma_{2^{l+1}D-1,2^l}+X_1+X_2+\cdots+X_{2m},
 \end{equation}
where $D, l$ and $m$ are positive integers, is trivially balanced.  In view of Theorem \ref{samepert}, the perturbation
\begin{equation}
 \sigma_{2^{l+1}D,2^l+1}+X_1+X_2+\cdots+X_{2m}
 \end{equation}
is also trivially balanced.
\end{theorem}

As it is the case of symmetric Boolean functions, computations suggest that trivially balanced perturbations are quite common.  Therefore, it is a good idea for us to study them in more detail.  Observe that 
Theorem \ref{trivialbalanced1} implies that once $\sigma_{n,k}+X_1$ is trivially balanced at one point, for example, when $n=2^{r}+k-1$, then $\sigma_{n,k}+X_1$ is trivially balanced for infinitely many 
$n$, i.e for $n=2^rm+k-1$ where $m$ is a positive integer.  This is not a coincidence, in fact, it is true for every perturbation. 

Let $1\leq k_1<\cdots<k_s$ be integers.  Consider 
$\sigma_{n+j,[k_1,\cdots, k_s]}+F({\bf X})$ where $F({\bf X})$ is a Boolean polynomial in the variables $X_1,\cdots, X_j$ ($j$ fixed).  Suppose that $\sigma_{n+j,[k_1,\cdots, k_s]}+F({\bf X})$ is trivially 
balanced, that is, it corresponds to a trivial solution $(\delta_0,\delta_1,\cdots,\delta_n)$ to (\ref{gensec}) over $\Gamma_j$.  For the sake of simplicity, suppose that 
$(\delta_0,\delta_1,\cdots,\delta_n)$ is a trivial solution of the form $\delta_{l} = -\delta_{n-l}$.   

The value of $\delta_l$ is given by
\begin{eqnarray}
\delta_l&=&\sum_{m=0}^j C_m(F)(-1)^{\sum_{i=0}^m \binom{m}{i}\left[\binom{l}{k_1-i}+\cdots+\binom{l}{k_s-i}\right]}\\\nonumber
&=&\sum_{m=0}^j C_m(F)(-1)^{\left(\binom{l+m}{k_1}+\cdots+\binom{l+m}{k_s}\right)},
\end{eqnarray}
and we know that
\begin{equation}
\binom{l+m}{k_1},\cdots,\binom{l+m}{k_s}  
\end{equation}
are all periodic modulo 2 with a period length $2^r$ where $r=\lfloor\log_2(k_s)\rfloor+1$.   This means that 
$$\delta_{l}=\delta_{l+2^r}\,\,\, \text{ and }\,\,\, \delta_{n-l}=\delta_{n-l+2^r},$$
but then,
$$\delta_l=-\delta_{n-l}=-\delta_{n+2^r-l}.$$
In other words, the tuple
$$(\delta_0,\delta_1,\cdots,\delta_{n+2^r})$$
is a trivial solution to (\ref{gensec}) over $\Gamma_{j}$.  However, this tuple corresponds to $S(\sigma_{n+2^r+j,[k_1,\cdots,k_s]}+F({\bf X}))$, i.e.
\begin{equation}
 S(\sigma_{n+2^r+j,[k_1,\cdots,k_s]}+F({\bf X})) = \sum_{l=0}^{n+2^r} (2\delta_l) \binom{n+2^r}{l}=0.
\end{equation}
This discussion leads to the following result.
\begin{theorem}
\label{trivialsols}
 Let $1\leq k_1<\cdots<k_s$ be integers and $F({\bf X})$ be a Boolean polynomial in the variables $X_1,\cdots,X_j$ ($j$ fixed).  Let $r=\lfloor \log_2(k_s) \rfloor+1$. Suppose that $n_0$ is a positive 
 integer such that 
\begin{equation}
\sigma_{n_0+j,[k_1,\cdots,k_s]}+F({\bf X})                    
\end{equation}
 is trivially balanced.  Then, 
 \begin{equation}
\sigma_{n_0+m\cdot 2^r+j,[k_1,\cdots,k_s]}+F({\bf X})  
 \end{equation}
 is trivially balanced for every non-negative integer $m$.
\end{theorem}
\begin{proof}
 This is a direct consequence of the above discussion.
\end{proof}

\begin{example}
 Consider the perturbation $\sigma_{23,8}+F({\bf X})$ where $$F({\bf X})=X_1+X_2+X_3+X_4+X_5+X_6+X_7+X_8+X_9.$$ This perturbation is actually balanced.  It corresponds to the equation
 \begin{equation}
  \sum_{l=0}^{14} \delta_l \binom{14}{l}=0,
 \end{equation}
where
 \begin{eqnarray*}
  {\boldsymbol{\delta}}&=& (\delta_0,\delta_1,\cdots,\delta_{14})\\
  &=& (-8, 28, -56, 70, -56, 28, -8, 0, 8, -28, 56, -70, 56, -28, 8).
 \end{eqnarray*}
Therefore, $\sigma_{23,8}+F({\bf X})$ is trivially balanced.  Theorem \ref{trivialsols} implies that $\sigma_{23 + 16 m,8}+F({\bf X})$ is trivially balanced for every non-negative integer $m$.
Now, since $F({\bf X})$ is balanced and $\sigma_{23 + 16 m,8}+F({\bf X})$ is balanced, then Theorem \ref{samepert} implies that the perturbation $\sigma_{24,9}+F({\bf X})$ is also balanced.  In fact, 
the perturbation is trivially balanced (as expected) and it corresponds to the equation
\begin{equation}
 \sum_{l=0}^{15}\delta_l\binom{15}{l}=0,
\end{equation}
where 
 \begin{eqnarray*}
  {\boldsymbol{\delta}}&=& (\delta_0,\delta_1,\cdots,\delta_{15})\\
  &=& (1, -9, 37, -93, 163, -219, 247, -255, 255, -247, 219, -163, 93, -37, 9, -1).
 \end{eqnarray*}
Theorem \ref{trivialsols} implies that $\sigma_{24 + 16 m,9}+F({\bf X})$ is trivially balanced for every non-negative integer $m$.
\end{example}

\begin{example}
Consider the perturbation $\sigma_{7, 4}+X_1+X_2$, which is trivially balanced.  The solution to (\ref{gensec}) over $\Gamma_2$ that corresponds to $S(\sigma_{7,4}+X_1+X_2)$ is 
    $$-\binom{5}{2}+\binom{5}{3}=0.$$
In view of Theorem \ref{trivialsols}, we conclude that $\sigma_{7+8m, 4}+X_1+X_2$ is trivially balanced for every non-negative integer $m$.  Now apply Theorem \ref{samepert} to conclude that 
$\sigma_{8+8m, 5}+X_1+X_2$ is balanced for every non-negative integer $m$. Observe that the solution to (\ref{gensec}) over $\Gamma_2$ that corresponds to $S(\sigma_{8,5}+X_1+X_2)$ is 
$$-\binom{6}{3}+2\binom{6}{4}-2\binom{6}{5}+2\binom{6}{6}=0,$$
which, at first, does not look like a trivial solution until we realize that it is equivalent to 
$$\binom{6}{0}-\binom{6}{1}+\binom{6}{2}-\binom{6}{3}+\binom{6}{4}-\binom{6}{5}+\binom{6}{6}=0.$$
Thus, $\sigma_{8, 5}+X_1+X_2$ is trivially balanced.  The same conclusion can be reached for $\sigma_{8+8m, 5}+X_1+X_2$.
\end{example}

\begin{example}
The reader can check that $\sigma_{15, [3,4,5,6]}+X_1+X_2+X_3$ is balanced and that its corresponding solution to (\ref{gensec}) over $\Gamma_3$ is 
\begin{eqnarray*}
  {\boldsymbol{\delta}}&=& (\delta_0,\delta_1,\cdots,\delta_{12})\\
  &=& (1, -2, 0, 2, -1, 0, 0, 0, 1, -2, 0, 2, -1).
 \end{eqnarray*}    
Therefore, $\sigma_{15+8m, [3,4,5,6]}+X_1+X_2+X_3$ is trivially balanced for every non-negative integer $m$.  This family of functions can be lifted up by Theorem \ref{samepertgen} to conclude that
$\sigma_{16+8m, [4,5,6,7]}+X_1+X_2+X_3$ is also balanced for every non-negative integer $m$. 
\end{example}

Let us recap what we have so far.  Exponential sums of perturbations of symmetric Boolean functions are connected to solutions of the equation (\ref{gensec}) over some bounded subsets of the integers.  As 
it is the case for symmetric Boolean functions, the concept of trivial and non-trivial solutions can be defined as well as the concepts of trivially balanced perturbation and sporadic balanced perturbation. 
Similar to the case of symmetric Boolean functions, computations suggest that many balanced perturbations turn out to be trivially balanced.  Moreover, Theorem \ref{trivialsols} states that once a perturbation of fixed degree 
is trivially balanced at one $n$, then it is trivially balanced for infinitely many $n$.  

With regard to this last theorem, in the case of a symmetric Boolean function of fixed degree, only trivially balanced cases exist when the number 
of variables grows.  This was first conjectured by Canteaut and Videau \cite{canteaut} and proved by Guo, Gao and Zhao \cite{ggz}.  In the next section, we establish a similar result for the type of 
perturbations considered in this article.  The techniques to be used are inspired by the ones presented in \cite{ggz}.


\section{Balancedness of perturbations as the number of variables grows}
\label{numvariablesgrows}

In this section we consider balancedness of perturbations when the number of variables grows.  This problem provides information about the converse of the statement of Theorem \ref{trivialsols}.  
The techniques used in this section are very similar to the ones presented in \cite{ggz}.  

Consider a perturbation of the form $\sigma_{n,[k_1,\cdots,k_s]}+F({\bf X})$ where $1\leq k_1<\cdots<k_s$ are integers and $F({\bf X})$ is a polynomial in the variables $X_1,\cdots,X_j$ with $j$ fixed.
The problem is to characterize when $S(\sigma_{n,[k_1,\cdots,k_s]}+F({\bf X}))=0$ for $n$ big enough.   Recall that the exponential sum $S(\sigma_{n,[k_1,\cdots,k_s]}+F({\bf X}))$ satisfies recurrence (\ref{mainrec}), that is, it satisfies
\begin{equation}
\label{mainrecagain}
x_n=\sum_{l=1}^{2^r-1}(-1)^{l-1}\binom{2^r}{l}x_{n-l},
\end{equation}
where $r=\lfloor\log_2(k_s)\rfloor+1$.  In other words, the sequence of exponential sums of our perturbation is a real solution to the linear recurrence (\ref{mainrecagain}).  Therefore, we first answer
the question of when a real solution $\{a_n\}$ to (\ref{mainrecagain}) is zero for some $n$ big enough.

The characteristic polynomial associated to (\ref{mainrecagain}) is 
\begin{equation}
 (X-2)\Phi_4(X-1)\Phi_8(X-1)\cdots\Phi_{2^r}(X-1).
\end{equation}
This implies that any solution $\{a_n\}$ to (\ref{mainrecagain}) has the form
\begin{equation}
\label{expform}
 a_n = d_0 \cdot 2^n + \sum_{l=1}^{2^r-1}d_l \lambda_l^n,
\end{equation}
where $\lambda_l=1+\xi_l^{-1}$ with $\xi_{l}=\exp\left(\frac{\mathrm{i}\pi l}{2^{r-1}}\right)$ and $\mathrm{i}=\sqrt{-1}$.  Observe that $\xi_{2^r-l}=\overline{\xi_l}$, $\lambda_{2^r-l}=\overline{\lambda_l}$,
and $\lambda_{2^{r-1}}=0$.  Moreover, if $\{a_n\}$ is a real solution, then we also have $d_{2^r-l}=\overline{d_l}$.   From now on, suppose that $\{a_n\}$ is a real solution to (\ref{mainrecagain}).

We now follow \cite{ggz} and express $\{a_n\}$ as
\begin{equation}
\label{realan}
 a_n = d_0 \cdot 2^n + 2\sum_{l=1}^{2^{r-1}-1}\text{Re}(d_l \lambda_l^n),
\end{equation}
where $\text{Re}(z)$ denotes the real part of $z$, and define
\begin{equation}
 t_l(n)=\text{Re}(d_l\lambda_l^n), \,\,\, \text{ for }0\leq l \leq 2^{r-1}-1.
\end{equation}
This leads to 
\begin{equation}
 a_n = t_0(n) + 2\sum_{l=1}^{2^{r-1}-1} t_l(n).
\end{equation}
The next lemma gives a characterization of when $a_n$ equals zero for some $n$ big enough.  It is basically the same statement as Lemma 3 of \cite{ggz}, which itself is a modification of a proof of Cai, Green and
Thierauf \cite{cai}. The proof follows almost verbatim as the one in \cite{ggz}.

\begin{lemma}
 Suppose that $\{a_n\}$ is a real solution to the linear recurrence (\ref{mainrecagain}).  Then there exists an integer $n_0$ such that for any $n>n_0$, 
 $$a_n=0\,\,\, \text{ if and only }\,\,\,t_l(n)=0, \,\,\,\text{ for }0\leq l\leq 2^{r-1}-1.$$
\end{lemma}

The number $t_l(n)$ can be expressed as
\begin{equation}
 t_l(n)=|d_l|\left(2\cos\left(\frac{\pi l}{2^r}\right)\right)^n\cos\left(\arg(d_l)-\frac{\pi n l}{2^r}\right).
\end{equation}
This implies that Remark 2 of \cite{ggz} also applies and therefore,
$$t_l(n)=0 \,\,\, \text{ if and only if } \,\,\, d_l=-\xi_l^n\overline{d_l}, \,\,\, \text{ for any }0\leq l \leq 2^{r-1}-1.$$
This characterizes the cases when a real solution $a_n$ to recurrence (\ref{mainrecagain}) is zero for $n$ big, that is, there exists an integer $n_0$ such that for any $n>n_0$, 
 $$a_n=0\,\,\, \text{ if and only }\,\,\,d_l=-\xi_l^n\overline{d_l}, \,\,\, \text{ for any }0\leq l \leq 2^{r-1}-1.$$

Let us go back to our perturbations.  Again, for simplicity, consider a perturbation of the form $\sigma_{n,k}+F({\bf X})$ where $F({\bf X})$ is a Boolean polynomial in the variables 
$X_1,\cdots, X_j$ with $j$ fixed.  
Recall that
\begin{equation}
\label{pertbinom}
S(\sigma_{n+j,k}+F({\bf X}))=\sum_{l=0}^n\left(\sum_{m=0}^j C_m(F)(-1)^{\sum_{i=0}^m \binom{m}{i}\binom{l}{k-i}}\right)\binom{n}{l}.
\end{equation}
Let $\delta_l^{(F)}(k)$ be defined as
\begin{equation}
\delta_l^{(F)}(k)=\sum_{m=0}^j C_m(F)(-1)^{\sum_{i=0}^m \binom{m}{i}\binom{l}{k-i}}.
\end{equation}
In other words, (\ref{pertbinom}) can be re-written as
\begin{equation}
S(\sigma_{n+j,k}+F({\bf X}))=\sum_{l=0}^n\delta_l^{(F)}(k) \binom{n}{l}.
\end{equation}
Observe that if we find a value of $n$ for which (\ref{pertbinom}) is balanced, then we find a solution to the Diophantine equation 
\begin{equation}
\label{dioph}
\sum_{l=0}^n \binom{n}{l} x_l=0,
\end{equation}
over $\Gamma_j^{(e)}$ and the solution would be given by $(\delta_0^{(F)}(k),\delta_2^{(F)}(k),\cdots, \delta_n^{(F)}(k))$.

Recall that every solution to (\ref{mainrecagain}) has the form (\ref{expform}).  The exponential sum of our perturbations, as well as the ones of symmetric Boolean functions, satisfy (\ref{mainrecagain}), therefore they can be
expressed in form (\ref{expform}).  In the case of the symmetric Boolean function $\sigma_{n,[k_1,\cdots,k_s]}$, where $1\leq k_1<\cdots<k_s$ are integers, Cai, Green and Thierauf \cite{cai} found an explicit formula for the coefficients in (\ref{expform}), that is, if $r=\lfloor \log_2(k_s)\rfloor+1$, then
\begin{equation}
S(\sigma_{n,[k_1,\cdots,k_s]})=c_0(k_1,\cdots,k_s)\cdot 2^n + \sum_{l=1}^{2^r-1}c_l(k_1,\cdots,k_s) \cdot \lambda_l,
\end{equation}
where 
\begin{equation}
c_l(k_1,\cdots,k_s)=\frac{1}{2^r} \sum_{i=0}^{2^r-1} (-1)^{\binom{i}{k_1}+\cdots+\binom{i}{k_s}} \xi_l^i.
\end{equation}
In view of equation (\ref{perturbationeq}), this implies that, for $r=\lfloor\log_2(k)\rfloor+1$, one has
\begin{equation}
S(\sigma_{n+j,k}+F({\bf X}))=d_0\cdot 2^n+\sum_{l=1}^{2^r-1}d_l \cdot \lambda_l^n,
\end{equation}
where
\begin{eqnarray*}
d_l &=& \sum_{m=0}^j C_m(F)  \left(\frac{1}{2^r}\sum_{a=0}^{2^r-1}(-1)^{\sum_{i=0}^m \binom{m}{i}\binom{a}{k-i}}\xi_l^a\right)\\
&=&\frac{1}{2^r}\sum_{a=0}^{2^r-1}\left(\sum_{m=0}^j C_m(F)(-1)^{\sum_{i=0}^m \binom{m}{i}\binom{a}{k-i}} \right)\xi_l^a\\
&=&\frac{1}{2^r}\sum_{a=0}^{2^r-1} \delta_{a}^{(F)}(k)\cdot \xi_l^a
\end{eqnarray*}
\begin{remark}
This discussion carries over to perturbations of the form $\sigma_{n,[k_1,\cdot,k_s]}+F({\bf X})$ without too much effort.
\end{remark}
With this information at hand, we are ready to present one of the main results of this article.  This result is a generalization of Canteaut and Videau's observation for symmetric Boolean functions
of fixed degree.  To be specific, we show that, excluding the trivial cases, balanced perturbations of fixed degree do not exist when the number of variables grows.
\begin{theorem}
Suppose that $1\leq k_1<\cdots<k_s$ are integers and $F({\bf X})$ is a Boolean polynomial in the variables $X_1,\cdots,X_j$ ($j$ fixed).  There exists an $n_0$ such that for every $n>n_0$,
$\sigma_{n+j,[k_1,\cdots,k_s]}+F({\bf X})$ is balanced if and only if it is trivially balanced.
\end{theorem}

\begin{proof}
We present the proof for the case of a perturbation of the form $\sigma_{n+j,k}+F({\bf X})$.  The general case follows by the same argument.  

Recall that $\{S(\sigma_{n+j,k}+F({\bf X}))\}$ is a real solution to (\ref{mainrecagain}).  Therefore, there exists an integer $n_0$ such that for any $n>n_0$, 
\begin{equation}
 S(\sigma_{n+j,k}+F({\bf X}))=0\,\,\, \text{ if and only if }\,\,\, d_l=-\xi_l^n\overline{d_l}, \,\,\, \text{ for any }0\leq l \leq 2^{r-1}-1,
\end{equation}
where
\begin{eqnarray}
d_l&=&\frac{1}{2^r}\sum_{a=0}^{2^r-1}\left(\sum_{m=0}^j C_m(F)(-1)^{\sum_{i=0}^m \binom{m}{i}\binom{a}{k-i}} \right)\xi_l^a\\\nonumber
&=&\frac{1}{2^r}\sum_{a=0}^{2^r-1} \delta_{a}^{(F)}(k)\cdot \xi_l^a,
\end{eqnarray}
and $r=\lfloor \log_2(k) \rfloor+1$.

Suppose that $n>n_0$ and that $d_l=-\xi_l^n\overline{d_l}$ for any $0\leq l \leq 2^{r-1}-1$.  Observe that
\begin{eqnarray}
 -\xi_l^n\overline{d_l}&=&-\frac{\xi_l^n}{2^r}\sum_{a=0}^{2^r-1}\left(\sum_{m=0}^j C_m(F)(-1)^{\sum_{i=0}^m \binom{m}{i}\binom{a}{k-i}} \right)\xi_l^{-a}\\\nonumber
 &=&-\frac{1}{2^r}\sum_{a=0}^{2^r-1}\left(\sum_{m=0}^j C_m(F)(-1)^{\sum_{i=0}^m \binom{m}{i}\binom{a}{k-i}} \right)\xi_l^{n-a}\\\nonumber
 &=&-\frac{1}{2^r}\sum_{t=n-2^r+1}^{n}\left(\sum_{m=0}^j C_m(F)(-1)^{\sum_{i=0}^m \binom{m}{i}\binom{n-t}{k-i}} \right)\xi_l^{t}\\\nonumber
 &=&-\frac{1}{2^r}\sum_{a=0}^{2^r-1}\left(\sum_{m=0}^j C_m(F)(-1)^{\sum_{i=0}^m \binom{m}{i}\binom{n-a}{k-i}} \right)\xi_l^{a}\\\nonumber
 &=&-\frac{1}{2^r}\sum_{a=0}^{2^r-1}\delta_{n-a}^{(F)}(k)\cdot\xi_l^{a},
\end{eqnarray}
where the previous to the last identity holds because $\left(\sum_{m=0}^j C_m(F)(-1)^{\sum_{i=0}^m \binom{m}{i}\binom{n-a}{k-i}} \right)\xi_l^{a}$ has period $2^r$.  Therefore,
$$d_l=-\xi_l^n\overline{d_l}\,\,\, \text{ for any }0\leq l \leq 2^{r-1}-1,$$ 
holds if and only if
\begin{equation}
\label{alwayszero}
 \sum_{a=0}^{2^r-1} (\delta_{a}^{(F)}(k)+\delta_{n-a}^{(F)}(k)) \xi_l^a=0, \,\,\, \text{ for any } 0\leq l \leq 2^{r-1}-1.
\end{equation}

Let $f(X)$ be the polynomial 
\begin{equation}
\label{thef}
f(X)=\sum_{a=0}^{2^r-1} (\delta_{a}^{(F)}(k)+\delta_{n-a}^{(F)}(k)) X^a.
\end{equation}
Observe that equation (\ref{alwayszero}) implies that all polynomials on the list
$$ X-1, X^2+1, X^4+1,\cdots, X^{2^{r-1}}+1$$
divide $f(X)$.  But all these polynomials are irreducible in $\mathbb{Q}[X]$, therefore
\begin{equation}
 (X-1)\prod_{t=1}^{r-1} (X^{2^t}+1)\,\, \text{ divides }\,\, f(X).
\end{equation}
However, the degree of both $(X-1)\prod_{t=1}^{r-1} (X^{2^t}+1)$ and $f(X)$ is $2^r-1$. Since $\mathbb{Q}[X]$ is a Unique Factorization Domain, then it follows that
\begin{equation}
\label{eqfcy1}
 F(X)=z\cdot (X-1)\prod_{t=1}^{r-1} (X^{2^t}+1),
\end{equation}
for some constant $z$ (it is not hard to see that $z$ must be an integer).  But
\begin{equation}
\label{eqcy2}
(X-1)\prod_{t=1}^{r-1} (X^{2^t}+1)= -1+X-X^2+X^3-X^4+\cdots-X^{2^r-2}+X^{2^r-1}.
\end{equation}
Equations (\ref{thef}), (\ref{eqfcy1}) and (\ref{eqcy2}) yield
\begin{equation}
\label{deltaa}
 \delta_{a}^{(F)}(k)+\delta_{n-a}^{(F)}(k) = (-1)^{a-1}z \text{ for }0\leq a\leq 2^r-1.
\end{equation}
Note that equation (\ref{deltaa}) holds beyond the range $0\leq a\leq 2^r-1$ because $\delta_{a}^{(F)}(k)$ has period $2^r$.

Thus, when $n$ is big enough, equation (\ref{deltaa}) characterizes all solutions
$$(\delta_0^{(F)}(k),\delta_1^{(F)}(k),\delta_2^{(F)}(k),\cdots, \delta_n^{(F)}(k))$$
to the Diophantine equation (\ref{dioph}) over the set $\Gamma_j^{(e)}$ that come from our perturbation. The next step is to show that all of them are trivial.

Suppose first that $n$ is odd, that is, suppose that $n=2m+1$.  We know that
\begin{equation}
 \delta_{a}^{(F)}(k)+\delta_{2m+1-a}^{(F)}(k) = (-1)^{a-1}z,
\end{equation}
where $z$ is a fixed integer.  Let $a=m$, then 
\begin{equation}
\label{forodd1}
 \delta_{m}^{(F)}(k)+\delta_{m+1}^{(F)}(k) = (-1)^{m-1}z.
\end{equation}
On the other hand, let $a=m+1$.  Then,
\begin{equation}
\label{forodd2}
\delta_{m+1}^{(F)}(k)+\delta_{m}^{(F)}(k) = (-1)^{m}z.
\end{equation}
Equations (\ref{forodd1}) and (\ref{forodd2}) imply that $z=0$.  But then
\begin{equation}
 \delta_{n-a}^{(F)}(k) = -\delta_{a}^{(F)}(k)
\end{equation}
and we conclude that the perturbation is trivially balanced when $n$ is odd.

Suppose now that $n$ is even, i.e. $n=2m$. As before, we know that 
\begin{equation}
\label{trivialQ}
 \delta_{a}^{(F)}(k)+\delta_{2m-a}^{(F)}(k) = (-1)^{a-1}z,
\end{equation}
where $z$ is a fixed integer. If $z=0$, then it is clear that the perturbation is trivially balanced.  Thus, suppose that $z\neq 0$.  Let $a=m$ in equation (\ref{trivialQ}).  Then,
\begin{equation}
 2\delta_{m}^{(F)} = (-1)^{m-1}z
\end{equation}
and therefore $z$ is even.  Say $z=2z_0$ with $z_0$ a non-zero integer.  Then, $\delta_{m}^{(F)} = (-1)^{m-1}z_0$ and
\begin{align}\nonumber
 (\delta_{0}^{(F)}(k),\delta_{1}^{(F)}(k),&\cdots, \delta_{2m}^{(F)}(k))\\
 &\sim(\delta_{0}^{(F)}(k)+\delta_{2m}^{(F)}(k),\delta_{1}^{(F)}(k)+\delta_{2m-1}^{(F)}(k),\cdots, \delta_{m-1}^{(F)}(k)+\delta_{m+1}^{(F)}(k), \delta_{m}^{(F)}(k),0,0,\cdots,0)\\\nonumber
 &\sim(2z_0,-2z_0,\cdots, (-1)^m2z_0, (-1)^{m-1}z_0,0,0,\cdots,0)\\\nonumber
 &\sim(2,-2,\cdots, (-1)^m2, (-1)^{m-1},0,0,\cdots,0)\\\nonumber
 &\sim(1,-1,1,-1,\cdots,-1,1).
\end{align}
We conclude that the perturbation is also trivially balanced when $n$ is even. 

Since it is clear that trivially balanced implies balanced, then we conclude that there is an integer $n_0$ such that for any $n>n_0$, the perturbation $\sigma_{n,k}+F({\bf X})$ 
is balanced if and only if it is trivially balanced.  This concludes the proof.

\end{proof}

\section{Some examples of sporadic balanced perturbations}
\label{sporadicpert}

In the previous section we learned that, excluding the trivial cases, balanced perturbations of fixed degree do not exist when the number of variables grows. Thus, as in the case for symmetric Boolean 
functions, it is of great interest to find sporadic balanced perturbations.  Next are some examples of sporadic balanced perturbations and their corresponding solutions to (\ref{gensec}).  
\begin{example}
  The perturbation $\sigma_{24,14}+X_1+X_2$ is balanced and sporadic.  The solution to (\ref{gensec}) over $\Gamma_2$ that corresponds to $S(\sigma_{24,14}+X_1+X_2)$ is 
  \begin{equation}
  \label{consequence1}
  \binom{22}{12}-\binom{22}{13}-\binom{22}{14}+\binom{22}{15}=0. 
  \end{equation}
  Observe that Theorem \ref{samepert} implies that $\sigma_{25,15}+X_1+X_2$ is also balanced and sporadic.  The solution to (\ref{gensec}) over $\Gamma_2$ that corresponds to $S(\sigma_{25,15}+X_1+X_2)$ is 
  \begin{equation}
  \label{particular3term}
   \binom{23}{13}-2\binom{23}{14}+\binom{23}{15}=0.
  \end{equation}
  Equation (\ref{particular3term}) corresponds to one solution to the following three consecutive binomial coefficients Diophantine equation
  \begin{equation}
  \label{3termequation}
   \binom{n}{k+2}-2\binom{n}{k+1}+\binom{n}{k}=0.
  \end{equation}
  Luca and Szalay \cite{luca} studied equations of the form
  \begin{equation}
   A\binom{n}{k}+B\binom{n}{k+1}+C\binom{n}{k+2}=0,
  \end{equation}
  where $A, B, C\in \mathbb{Z}$, $A>0$, $C\neq 0$ and $\gcd(A,B,C)=1$.  As part of their study, they provided infinitely many solutions to (\ref{3termequation}) given by 
  \begin{equation}
   \label{general3term}
   \binom{t^2-2}{(t^2+t-4)/2}-2\binom{t^2-2}{(t^2+t-2)/2}+\binom{t^2-2}{(t^2+t)/2}=0,
  \end{equation}
  where $t$ is any integer satisfying $|t|\geq 3$.  Observe that (\ref{particular3term}) corresponds to (\ref{general3term}) when $t=5$.  Moreover, (\ref{consequence1}) can be obtained from (\ref{particular3term})
  by applying the identity
  \begin{equation}
   \binom{n}{k}=\binom{n-1}{k}+\binom{n-1}{k-1}.
  \end{equation}
\end{example}
  
 \begin{example}
 In \cite{singmaster}, Singmaster shows that the Diophantine equation
 \begin{equation}
 \binom{n}{k}+ \binom{n}{k+1}= \binom{n}{k+2},
 \end{equation}
 has infinitely many solutions given by $n=F_{2i+2}F_{2i+3}-1$ and $k=F_{2i}F_{2i+3}-1$, where $F_n$ represent the $n$-th Fibonacci number.  The smallest of these solutions is given by
 \begin{eqnarray}
 \label{firstFibo}
\binom{F_{4}F_{5}-1}{F_{2}F_{5}-1}+ \binom{F_{4}F_{5}-1}{F_{2}F_{5}}- \binom{F_{4}F_{5}-1}{F_{2}F_{5}+1}&=&  0\\\nonumber
 \binom{14}{4}+ \binom{14}{5}-\binom{14}{6} &=&  0.
 \end{eqnarray}
 Observe (\ref{firstFibo}) is a non-trivial solution to (\ref{gensec}) over $\Gamma_1$.  A nice problem is to find sporadic balanced perturbations of the form $\sigma_{15,[k_1,\cdots,k_s]}+X_1$ that
 corresponds to  (\ref{firstFibo}).  There are only four of such perturbations of degree less than or equal to 14.  The first one is $\sigma_{15,[5,6,10,12,13]}+X_1$, which corresponds to the equivalent 
 solution
 \begin{equation}
 \binom{14}{4}-\binom{14}{6}+\binom{14}{9}=0.
 \end{equation}
 Another is $\sigma_{15,[6, 8, 9, 10,13]}+X_1$, which corresponds to the equivalent solution 
 \begin{equation}
 \binom{14}{5}-\binom{14}{8}+\binom{14}{10}=0.
 \end{equation}
 The third one is $\sigma_{15,[6, 7, 11, 12,14]}+X_1$ and it corresponds to the equivalent solution 
 \begin{equation}
 \binom{14}{5}-\binom{14}{6}+\binom{14}{10}=0.
 \end{equation}
 Finally, the last one is $\sigma_{15,[5, 6, 7, 8, 9, 11,14]}+X_1$.  The corresponding equivalent solution is
 \begin{equation}
 \binom{14}{4}-\binom{14}{8}+\binom{14}{9}=0.
 \end{equation} 
 \end{example} 
 
\begin{example}
 Let us go back to the equation (\ref{general3term}).  Note that the smallest solution to this equation is given by
 \begin{equation}
 \label{smallsporadic}
  \binom{7}{4}-2\binom{7}{5}+\binom{7}{6}=0.
 \end{equation}
 Table \ref{perturbationsks7} includes all sporadic balanced perturbations of the form $\sigma_{8,[k_1,\cdots, k_s]}+X_1$ for $k_s\leq 7$ for which their corresponding solutions to the Diophantine 
 equation (\ref{gensec}) are equivalent to (\ref{smallsporadic}).
 \begin{table}[h!]
 \caption{Perturbations and their corresponding solutions to (\ref{gensec}).}
 \begin{tabular}{|l|l|}
 \hline
 Perturbation & Corresponding solution\\
 \hline
  $\sigma_{8,[3,6]}+X_1$ &   $(0, 0, 1, -1, 0, 1, -1, 0)$\\
  $\sigma_{8,[1,2,6]}+X_1$ &  $(1, 0, -1, 0, 1, -1, 1, -1)$\\
  $\sigma_{8,[1,5,6]}+X_1$ &   $(1, -1, 1, -1, 0, 1, 0, -1)$\\
  $\sigma_{8,[2,3,5,6]}+X_1$ &  $(0, 1, -1, 0, 1, -1, 0, 0)$\\
  $\sigma_{8,[1,4,7]}+X_1$ &   $(1, -1, 1, 0, -1, 1, 0, -1)$\\
  $\sigma_{8,[2, 3, 4, 7]}+X_1$ &   $(0, 1, -1, 1, 0, -1, 0, 0)$\\
  $\sigma_{8,[3, 4, 5, 7]}+X_1$ &   $(0, 0, 1, 0, -1, 1, -1, 0)$\\
  $\sigma_{8,[1, 2, 4, 5, 7]}+X_1$ &  $(1, 0, -1, 1, 0, -1, 1, -1)$\\
  \hline
 \end{tabular}
 \label{perturbationsks7}
 \end{table}

 \noindent
 Table \ref{perturbationsks8} includes all sporadic balanced perturbations of the form $\sigma_{9,[k_1,\cdots, k_s]}+X_1+X_2$ for $k_s\leq 8$ for which their corresponding solutions to the Diophantine equation (\ref{gensec}) are 
 equivalent to (\ref{smallsporadic}).
 
 \begin{table}[h!]
 \caption{Perturbations and their corresponding solutions to (\ref{gensec}).}
 \begin{tabular}{|l|l|}
 \hline
 Perturbation & Corresponding solution\\
 \hline
  $\sigma_{9,[3,6]}+X_1+X_2$ &   $(0, 0, 1, -1, 0, 1, -1, 0)$\\
  $\sigma_{9,[3,7]}+X_1+X_2$ & $(0, -1, 2, -1, 0, 0, 0, 0)$\\
  $\sigma_{9,[6, 7]}+X_1+X_2$ & $(0, 0, 0, 0, -1, 2, -1, 0)$\\
  $\sigma_{9,[1, 3, 7]}+X_1+X_2$ & $(2, -1, 0, -1, 2, -2, 2, -2)$\\
  $\sigma_{9,[1, 4, 7]}+X_1+X_2$ & $(2, -2, 1, 1, -2, 1, 1, -2)$\\
  $\sigma_{9,[1, 6, 7]}+X_1+X_2$ & $(2, -2, 2, -2, 1, 0, 1, -2)$\\
  $\sigma_{9,[1, 2, 3, 7]}+X_1+X_2$ & $(1, 0, 1, -2, 1, 1, -1, -1)$\\
  $\sigma_{9,[1, 2, 4, 7]}+X_1+X_2$ & $(1, 1, -2, 2, -1, 0, 0, -1)$\\
  $\sigma_{9,[1, 2, 6, 7]}+X_1+X_2$ & $(1, 1, -1, -1, 2, -1, 0, -1)$\\
  $\sigma_{9,[1, 3, 4, 6, 7]}+X_1+X_2$ & $(2, -1, -1, 2, -1, -1, 2, -2)$\\
  $\sigma_{9,[1, 2, 3, 4, 6, 7]}+X_1+X_2$ & $(1, 0, 0, 1, -2, 2, -1, -1)$\\
  $\sigma_{9,[5,8]}+X_1+X_2$ &  $(0, 0, 0, -1, 2, -2, 1, 0)$\\
  $\sigma_{9,[2, 5, 8]}+X_1+X_2$ & $(-1, 1, 1, -2, 1, 1, -2, 1)$\\
  $\sigma_{9,[3, 4, 5, 8]}+X_1+X_2$ & $(0, -1, 1, 1, -2, 1, 0, 0)$\\
  $\sigma_{9,[3, 5, 6, 8]}+X_1+X_2$ & $(0, -1, 2, -2, 1, 0, 0, 0)$\\
  $\sigma_{9,[4, 5, 6, 8]}+X_1+X_2$ & $(0, 0, -1, 2, -1, -1, 1, 0)$\\
  $\sigma_{9,[2, 3, 4, 5, 8]}+X_1+X_2$ & $(-1, 2, -2, 2, -1, 0, -1, 1)$\\
  $\sigma_{9,[2, 3, 5, 6, 8]}+X_1+X_2$ & $(-1, 2, -1, -1, 2, -1, -1, 1)$\\
  $\sigma_{9,[2, 4, 5, 6, 8]}+X_1+X_2$ & $(-1, 1, 0, 1, -2, 2, -2, 1)$\\
  \hline
 \end{tabular}
 \label{perturbationsks8}
 \end{table}
 \end{example}
Similar examples can be produced with the aid of computers.  A {\em Mathematica} implementation can be found in 
$$\texttt{http://emmy.uprrp.edu/lmedina/papers/diophpert/}.$$
For example, using this implementation, we found that there are 265 sporadic balanced perturbations of the form $\sigma_{n,[k_1,\cdots,k_s]}+X_1$ with $n, k_s\leq 17$.  Also, there are 606 sporadic balanced 
perturbations of the form $\sigma_{n,[k_1,\cdots,k_s]}+X_1+X_2$ with $n, k_s\leq 17$.  The reader is invited to use this Mathematica implementation to find more.

\medskip \medskip
\noindent
{\bf Acknowledgments.} 
The second author was partially supported as a student by NSF-DUE 1356474 and the
Mellon-Mays Undergraduate Fellowship.

\bibliographystyle{plain}

\end{document}